\documentclass{epsconf}

\usepackage{amsmath,booktabs}

\usepackage{algorithm,algorithmicx,algpseudocode}

\renewcommand{\subset}{\subseteq}

\newcommand{\dx}{\,\text{\rmfamily{}\upshape{}d}x}

\def\N{\mathbb  N}
\def\R{\mathbb  R}
\def\Z{\mathbb  Z}

\def\argmin{\operatorname*{arg\, min}}

\def\dtol{\delta_{\mathrm{tol}}}

\usepackage{siunitx}
\title{A numerical solution approach for non-smooth optimal control problems
based on the Pontryagin maximum principle}
\author{Daniel Wachsmuth\authormark{1}}

\affil[1]{\url{daniel.wachsmuth@uni-wuerzburg.de}\\ University of W\"urzburg}

\begin{document}

\maketitle

\begin{abstract}
We consider nonsmooth optimal control problems
subject to a linear elliptic partial differential equation
with homogeneous Dirichlet boundary conditions.
It is well-known that local solutions satisfy the celebrated Pontryagin maximum principle.
In this note, we will investigate an optimization method that is based on the maximum principle.
We prove that the discrepancy in the maximum principle vanishes along the resulting sequence of iterates.
Numerical  experiments confirm the theoretical findings.
\end{abstract}

\section{Introduction}

In this note, we consider the following optimal control problem:
Minimize
\begin{equation}\label{eq_P}
 J(y,u) := \frac12\|y-y_d\|_{L^2(\Omega)}^2 + \int_\Omega g(u(x)) \dx
\end{equation}
over all $u\in L^2(\Omega)$ and $y\in H^1_0(\Omega)$ satisfying
\[\begin{aligned}
 -\Delta y &= u &&\text{ in } \Omega,\\
 y &= 0 &&\text{ on } \partial\Omega.
\end{aligned}\]
Here,  $\Omega\subset \R^d$ is a bounded domain, and $g: \R \to \bar \R = \R \cup \{+\infty\}$ is assumed to be proper and lower semicontinuous.
In addition, we require
\begin{equation}\label{eq_ass_g}
 \lim_{ |v| \to \infty} \frac{ g(v) }{|v|} = +\infty.
\end{equation}
Note, that we assume neither continuity nor convexity of $g$.
Hence, it is impossible to prove existence of solutions of \eqref{eq_P}.
In fact, one can construct problems without solution, see \cite[Section 4.5]{DWachsmuth2019}.
In this note, we will work with the example
\begin{equation}\label{eq_def_g}
 g(u):= \frac\alpha2 u^2 + I_{\mathbb Z} (u)
 = \begin{cases} \frac\alpha2 u^2 & \text{ if } u \in \mathbb Z\\
                              +\infty & \text{ otherwise,}
                             \end{cases}
\end{equation}
where $\alpha>0$.
If $g$ is assumed to be convex and continuous, then existence of  solutions of \eqref{eq_P} can be proven by the direct method of the calculus of variations \cite{Troltzsch2010}.
Let us remark that by the above assumptions $g$ is bounded from below.

If solutions exist, then the Pontryagin maximum principle \cite{PontryaginBoltyanskiiGamkrelidzeMishchenko1962} is a necessary optimality condition.
Its main feature is that no differentiability with respect to the controls is needed, and so it is perfectly suited for the problems considered here.
In fact, due to the structure of the problem (linear state equation, convexity of $J$ with respect to $y$), the maximum principle is sufficient.
We refer to \cite{BonnansCasas1995,Casas1994,CasasRaymondZidani2000,RaymondZidani1998} for the Pontryagin maximum principle applied to optimal control
problems for partial differential equations.
The goal of this note is to construct an algorithm to solve the maximum principle.
We will comment on related work in Section \ref{sec_algorithm}.

\section{Sensitivity analysis}

In this section, we will perform a sensitivity analysis with respect to perturbations of the control with characteristic functions.
The setup is as follows: Let
$u, \tilde u\in L^2(\Omega)$
be feasible controls, i.e., the integrals $ \int_\Omega g(u) \dx$ and $ \int_\Omega g(\tilde u) \dx $ exist.
Let $B \subset \Omega$ be measurable.
We define
\[
u_B := u + \chi_B (\tilde u-u).
\]
Let $y,y_B$ be the uniquely determined weak solutions of
 \[\begin{aligned}
 -\Delta y &= u &  -\Delta y_B &= u_B &&\text{ in } \Omega,\\
 y &= 0         &  y_B &= 0 &&\text{ on } \partial\Omega.
\end{aligned}\]
Let $p\in H^1_0(\Omega)$ be the weak solution of the adjoint equation
\[\begin{aligned}
  -\Delta p &= y-y_d &&\text{ in } \Omega,\\
  p &= 0 &&\text{ on } \partial\Omega.
\end{aligned}\]
The goal is now to estimate $ J(y_B,u_B) - J(y,u)$ in terms of $ u, \tilde u,p$ and the Lebesgue measure $|B|$ of $B$.
Here, we have the following result.

\begin{lemma}\label{lem_21}
 Under the assumptions above, we have
 \[
 J(y_B,u_B) - J(y,u) =   \int_B  (\tilde u - u)p + g(\tilde u) - g(u)\dx + \frac12 \|y_B-y\|_{L^2(\Omega)}^2 .
 \]
\end{lemma}
\begin{proof}
This follows directly from the definition of $p$ and $u_B$:
\[\begin{aligned}
 J(y_B,u_B) - J(y,u) & = \frac12\|y_B-y_d\|_{L^2(\Omega)}^2 + \int_\Omega g(u_B) \dx -\frac12\|y-y_d\|_{L^2(\Omega)}^2 - \int_\Omega g(u) \dx\\
 & = \int_\Omega (y_B - y)(y-y_d) + \frac12 (y_B-y)^2 \dx + \int_B g(\tilde u) - g(u) \dx \\
%  &= \int_\Omega (u_B - u)p + \frac12 (y_B-y)^2 \dx + \int_B g(\tilde u) - g(u) \dx \\
 & = \int_B  (\tilde u - u)p + g(\tilde u) - g(u)\dx + \frac12 \|y_B-y\|_{L^2(\Omega)}^2.
  \end{aligned}
\]
\end{proof}

We will now prove that $\|y_B-y\|_{L^2(\Omega)}^2$ is of higher order with respect to the Lebesgue measure $|B|$ of $B$.

\begin{lemma}
There are constants $c>0$ and $\nu>1/2$ independent of $u,\tilde u,B$ such that
 \[
  \|y_B-y\|_{L^2(\Omega)} \le c \, |B|^{\nu} \cdot  \| \tilde u-u\|_{L^\infty(\Omega)},
 \]
where $|B|$ denotes the Lebesgue measure of $B$. The constant $\nu$ can be chosen as
\[
 \nu = \begin{cases} 1 & \text{ if } d\le 3, \\
       1-\epsilon & \text{ if } d=4 \text{ for }\epsilon>0, \\
        \frac12+\frac2d & \text{ if } d>4.
        \end{cases}
        \]
\end{lemma}
\begin{proof}
We prove the claim by a well-known duality argument.
Assume $d\le3$.
Let $w\in L^2(\Omega)$ be given.
Let $z,q  \in H^1_0(\Omega)$ be the weak solutions of
 \[\begin{aligned}
 -\Delta z &= w &-\Delta q &= z &&\text{ in } \Omega,\\
 z &= 0 &q &= 0 &&\text{ on } \partial\Omega.
\end{aligned}\]
Due to \cite{Stampacchia1965}, there is  $c>0$ independent of $w,z$ such that
\[
 \|z\|_{L^\infty(\Omega)} \le c \|w\|_{L^2(\Omega)}.
\]
Testing the weak formulations with $z$ and $q$ yields
\[
 \|z\|_{L^2(\Omega)}^2  = \int_\Omega wq\dx \le \|w\|_{L^1(\Omega)}\|q\|_{L^\infty(\Omega)} \le c \|w\|_{L^1(\Omega)}\|z\|_{L^2(\Omega)}.
\]
This proves $\|z\|_{L^2(\Omega)} \le c \|w\|_{L^1(\Omega)}$. Applying this estimate to $z:= y_B - y$ and $w:= u_B-u$ yields the claim with
\[
\|y_B-y\|_{L^2(\Omega)} \le c \| u_B-u\|_{L^1(\Omega)} \le c \, |B|  \cdot \| \tilde u-u\|_{L^\infty(\Omega)} .
\]
In case $d>3$ one can use the estimates from \cite[Theorem 18]{BrezisStrauss1973}.
\end{proof}

Combining these results proves the following theorem.

\begin{theorem}\label{thm_sensitivity}
Let $u, \tilde u\in L^\infty(\Omega)$. Let $B \subset \Omega$ be measurable.
Let $\tilde u, y_B,y,p$ be defined as above. Then there are $\gamma >0$ and $c>0$ independent of $u,\tilde u,B$ such that
\[
 J(y_B,u_B) - J(y,u) \le   \int_B  (\tilde u - u)p + g(\tilde u) - g(u)\dx + c\, |B|^{1+\gamma}\| \tilde u-u\|_{L^\infty(\Omega)}^2.
\]
\end{theorem}

\section{Pontryagin maximum principle}

With the help of  Theorem \ref{thm_sensitivity} we can prove the Pontryagin maximum principle.

\begin{theorem}\label{thm_pmp}
 Let $\bar u\in L^\infty(\Omega)$ be locally optimal with respect to $L^1(\Omega)$ topology for the control problem \eqref{eq_P}. Let $\bar y,\bar p \in H^1_0(\Omega)$ be the optimal state and adjoint solving
 \[\begin{aligned}
   -\Delta \bar y &= \bar u & -\Delta \bar p &= \bar y-y_d &&\text{ in } \Omega,\\
  \bar y &= 0 &\bar p &= 0 &&\text{ on } \partial\Omega.
 \end{aligned}\]
 Let $v\in \mathbb \R$ be such that $g(v)<+\infty$. Then
 \begin{equation}\label{eq_pmp}
  \bar u(x) \bar p(x) + g(\bar u(x)) \le v  \bar p(x) + g(v) \text{ for almost all }x \in \Omega.
 \end{equation}
\end{theorem}
\begin{proof}
 Let $v\in \mathbb \R$ be such that $g(v)<+\infty$. Applying Theorem \ref{thm_sensitivity} with $u:=\bar u$, $\tilde u:=v$ yields
 \[
  0 \le J(y_B,u_B) - J(\bar y, \bar u) =   \int_B  (v- \bar u)\bar p + g(\tilde u) - g(\bar u)\dx + o( |B|).
 \]
By standard arguments based on the Lebesgue differentiation theorem, see, e.g., \cite[Theorem 2.1]{NatemeyerWachsmuth2021}, the claim follows.
\end{proof}

The maximum principle is a sufficient condition for the problem considered here.

\begin{corollary}\label{cor_pmp_suff}
Let
$\bar u\in L^2(\Omega)$
satisfy the conclusion \eqref{eq_pmp} of Theorem \ref{thm_pmp}. Then $\bar u$ is global optimal for \eqref{eq_P}.
\end{corollary}
\begin{proof}
 Let
 $\tilde u\in L^2(\Omega)$
 be an admissible control with associated state $\tilde y$. Then Lemma \ref{lem_21} with $B=\Omega$ yields
 \[
  J(\tilde y, \tilde u) - J(\bar y, \bar u) =   \int_\Omega  (\tilde u - \bar u)\bar p + g(\tilde u) - g(\bar u)\dx + \frac12 \|\tilde y-\bar y\|_{L^2(\Omega)}^2.
 \]
 Since $\bar u$ satisfies \eqref{eq_pmp}, the first expression is non-negative, which implies $J(\tilde y, \tilde u) - J(\bar y, \bar u) \ge \frac12 \|\tilde y-\bar y\|_{L^2(\Omega)}^2 \ge0$.
\end{proof}

\section{Construction of an algorithm}
\label{sec_algorithm}

We will now apply Theorem \ref{thm_sensitivity} with $u := u_k$ and $\tilde u:=\tilde u_k$, where $u_k$ is the current iterate of the algorithm to be devised.
Let $y_k$ and $p_k$ be the associated state and adjoint.
The control $\tilde u_k$ has to be computed in each iteration. Let $B_k$ be measurable.
Then we have
\begin{equation}\label{eq_sens_uk}
  J(y_{B_k},u_{B_k}) - J(y_k,u_k) = \int_{B_k}  (\tilde u_k - u_k)p_k + g(\tilde u) - g(u_k)\dx  + o(|B_k|).
\end{equation}
The idea is now to choose $\tilde u_k$ and $B_k$ such that $ J(y_{B_k},u_{B_k}) - J(y_k,u_k) $ is negative and
to define the new iterate by
\[
 u_{k+1} = u_k + \chi_{B_k} ( \tilde u_k - u_k).
\]
In view of the maximum principle, Theorem \ref{thm_pmp},
it is natural to choose $\tilde u_k$ as a function satisfying
\begin{equation}\label{eq_def_tilde_uk}
 \tilde u_k(x) \in \argmin_{v \in \R} v p_k + g(v) .
\end{equation}
In addition, $B_k$ will be chosen to get sufficient descent.

Let us comment on related work. The classic algorithm of \cite{KrylovCernousprimeKo1962} chooses $B_k := \Omega$, resulting in a fixed-point scheme
to solve the maximum principle. The min-h method of \cite{Gottlieb1967} uses the update $u_{k+1} := u_k + t( \tilde u_k - u_k)$ with $t\in (0,1]$,
and is thus only suited for convex functions $g$.
In the monograph  \cite{Srochko2000}, a method similar to ours is presented to solve optimal control problems with ODEs.
Let us also also mention the review papers \cite{ChernousprimeKoLyubushin1982,Strekalovsky2014}.
In \cite{MannsHahnKirchesLeyfferSager2023} binary control problems are solved with a similar approach: there a trust-region globalization is proposed,
whereas we use an Armijo line-search to globalize.

As motivated above, we will compute $\tilde u_k$ as a result of the pointwise minimization
\[
 \tilde u_k(x) \in \argmin_{v \in \R} v p_k + g(v) .
\]
Due to \eqref{eq_ass_g} this problem is solvable for all $x$.
A measurable selection of this argmin-map exists \cite{AubinFrankowska1990}.
For the example of $g$ proposed in \eqref{eq_def_g}, we get
\[
 \tilde u_k(x) \in \operatorname{round} \left( -\frac1\alpha p_k(x) \right).
\]
It remains to describe how $B_k$ is chosen.
Here, we are faced with two competing goals:
In order to make the first term in \eqref{eq_sens_uk} as small as possible, $B_k$ has to be chosen as large as possible.
However, to control the remainder term in \eqref{eq_sens_uk}, $|B_k|$ has to be chosen sufficiently small.

We propose the following line-search. Given $t\in (0,1]$, choose $B_t$ such that
\begin{equation}\label{eq_cond_Bk}
 \begin{aligned}
   \int_{B_t} (\tilde u - u_k)p_k + g(\tilde u) - g(u_k)\dx &\le t  \int_\Omega (\tilde u - u_k)p_k + g(\tilde u) - g(u_k)\dx, \\
    |B_t| &\le t \cdot |\Omega|.
 \end{aligned}
\end{equation}
Due to the celebrated Lyapunov convexity theorem, see, e.g., \cite[Theorem 5.5]{Rudin1973}, a measurable set $B_t$ satisfying
\eqref{eq_cond_Bk} exists.
Given $t$ and $B_t$, we set $u_t := u_k + \chi_{B_t} ( \tilde u_k - u_k)$. Let $y_t$ be the associated state.

The parameter $t_k$ is determined by the following procedure: Let $t_k$ be the largest number in $\{ \beta^l : \ l \in \N \cup \{0\} \}$, where $\beta \in (0,1)$,
that satisfies the descent condition
\begin{equation}\label{eq_cond_tk}
 J(y_t, u_t) - J(y_k,u_k) \le \sigma  \int_{B_t} (\tilde u - u_k)p_k + g(\tilde u) - g(u_k)\dx
\end{equation}
where $\sigma\in(0,1)$, and $B_t$ is a measurable set satisfying \eqref{eq_cond_Bk}. This condition is inspired by the well-known Armijo line-search in nonlinear optimization.
If $u_k$ does not satisfy the maximum principle, there is an admissible step-size $t_k$, and the resulting algorithm produces a new iterate with smaller value of the objective.

\begin{lemma} \label{lem_line_search_welldefined}
Suppose that
\[
\int_\Omega (\tilde u - u_k)p_k + g(\tilde u) - g(u_k)\dx < 0 .
\]
 There is $t_0>0$ such that for all $t\in (0,t_0)$ condition \eqref{eq_cond_tk} is satisfied.
\end{lemma}
\begin{proof}
Due to Theorem \ref{thm_sensitivity}, we have
\begin{multline*}
 J(y_t, u_t) - J(y_k,u_k) - \sigma  \int_{B_t} (\tilde u - u_k)p_k + g(\tilde u) - g(u_k)\dx\\
 \le (1-\sigma)   \int_{B_t} (\tilde u - u_k)p_k + g(\tilde u) - g(u_k)\dx + o(t)\\
 \le t  (1-\sigma) \int_\Omega (\tilde u - u_k)p_k + g(\tilde u) - g(u_k)\dx + o(t),
 \end{multline*}
which proves the claim.
\end{proof}

The resulting algorithm is sketched in Algorithm \ref{alg_top}.

\begin{algorithm}[htbp]
\begin{algorithmic}
 \State Choose $\beta\in (0,1)$, $\sigma\in(0,1)$, $u_0$ with $\int_\Omega g(u_0)\dx<\infty$,  $\dtol\ge0$. Set $k:=0$.

 \Loop \Comment{Gradient descent}

 \State Compute state $y_k$ and adjoint $p_k$ associated to $u_k$.

 \State Compute $\tilde u_k$ as in \eqref{eq_def_tilde_uk}.
%  \State
 \If{$\left| \int_\Omega (\tilde u_k - u_k)p_k + g(\tilde u) - g(u_k)\dx \right|\le \dtol$}  \Comment{Termination criterion}
	\State  \textbf{return} $u_k$
 \EndIf

 \State $t:=1$.

 \Loop \Comment{Armijo line-search}
 \State Compute $B_{k,t}$ satisfying \eqref{eq_cond_Bk}.
 \State Compute $J(y_t,u_t)$.
 \If{ \eqref{eq_cond_tk} is satisfied } \State \textbf{break}\EndIf
 \State $t:= \beta \cdot t$.

 \EndLoop
 \State $t_k := t$.  \Comment{Update}
 \State $u_{k+1} := u_k + \chi_{B_{k,t_k}} ( \tilde u_k - u_k)$.
\State $k:=k+1$.
\EndLoop
\end{algorithmic}
\caption{Maximum-principle based descent algorithm}
\label{alg_top}
\end{algorithm}

Let us now turn to the convergence analysis of Algorithm \ref{alg_top}.
Here, we follow the related analysis in \cite{DWachsmuth2024}.
Let us define
\[
 \rho_k := \int_\Omega (\tilde u_k - u_k)p_k + g(\tilde u) - g(u_k)\dx.
\]
Due to the choice of $\tilde u_k$ in \eqref{eq_def_tilde_uk}, it follows $\rho_k \le0$.
If $\rho_k=0$ then $u_k$ satisfies the maximum principle Theorem \ref{thm_pmp}, and the corresponding control $u_k$ is optimal by Corollary \ref{cor_pmp_suff}.

\begin{lemma} \label{lem_conv_aux1}
 Let $(u_k)$ be an infinite sequence generated by Algorithm \ref{alg_top}. Then
  \[
  \sum_{k=0}^\infty t_k \|\rho_k\|_{L^1(\Omega)} < +\infty.
  \]
\end{lemma}
\begin{proof}
 Using conditions \eqref{eq_cond_tk} and \eqref{eq_cond_Bk} shows
 \[
  J(y_{k+1}, u_{k+1}) - J(y_k,u_k) \le \sigma  \int_{B_{t_k}} (\tilde u - u_k)p_k + g(\tilde u) - g(u_k)\dx \le t_k  \int_\Omega (\tilde u - u_k)p_k + g(\tilde u) - g(u_k)\dx
  = - t_k \|\rho_k\|_{L^1(\Omega)}.
\]
Due to \eqref{eq_ass_g}, $g$ has a global minimum and is bounded from below, so that $J$ is bounded from below by some $M\in \R$.
Summing this inequality over $k\in \N$ and using $J \ge M$ proves $\sum_{k=1}^\infty t_k \|\rho_k\|_{L^1(\Omega)} \le J(y_0,u_0) -M< \infty$.
\end{proof}

For simplicity, we assume for the subsequent convergence analysis that
\begin{equation} \label{eq_ass_g_compact}
 \operatorname{dom} g := \{v: \ g(v) < \infty\}
\end{equation}
is compact. Then the set of iterates $(u_k)$ and $(\tilde u_k)$ is uniformly bounded in $L^\infty(\Omega)$.

\begin{corollary} \label{cor_uk_bounded}
Assume \eqref{eq_ass_g_compact}.
 Let $M>0$ such that $  \operatorname{dom} g  \subset [-M,+M]$. Then $\|u_k\|_{L^\infty(\Omega)} \le M$ and $\|\tilde u_k\|_{L^\infty(\Omega)} \le M$ for all $k$.
\end{corollary}

\begin{theorem} Assume \eqref{eq_ass_g_compact}.
 Either the Algorithm \ref{alg_top}  stops after finitely many steps
 with
 \[
 \left| \int_\Omega (\tilde u_k - u_k)p_k + g(\tilde u) - g(u_k)\dx \right|\le \dtol
 \]
 (so that $u_k$ satisfies the maximum principle if $\dtol=0$), or
 \[
  \int_\Omega (\tilde u - u_k)p_k + g(\tilde u) - g(u_k)\dx \to 0,
 \]
 i.e., the residual in the maximum principle tends to zero,
and $(u_k)$ is a minimizing sequence.
\end{theorem}
\begin{proof}
We follow the proof of the related result \cite[Theorem 6.7]{DWachsmuth2024}.
Let us assume the algorithm generates an infinite sequence of iterates.
 Let $k$ be such that $t_k < 1$. Due to the line-search procedure of Algorithm \ref{alg_top},
 it follows that $t:=\beta^{-1} t_k \le 1$ violates the descent condition \eqref{eq_cond_tk}, that is
 \[
 0<  J(y_t, u_t) - J(y_k,u_k) - \sigma  \int_{B_t} (\tilde u - u_k)p_k + g(\tilde u) - g(u_k)\dx.
 \]
As in the proof of Lemma \ref{lem_line_search_welldefined}, we get from Theorem \ref{thm_sensitivity}
 \[
 0<  t  (1-\sigma) \int_\Omega (\tilde u - u_k)p_k + g(\tilde u) - g(u_k)\dx + c\, |t|^{1+\gamma}\| \tilde u-u\|_{L^\infty(\Omega)}.
 \]
 Together with Corollary \ref{cor_uk_bounded}, we get
 \[
  0 < - t (1-\sigma) \|\rho_k\|_{L^1(\Omega)} + c |t|^{1+\gamma},
 \]
where $c$ is independent of $k$. This implies
\[
 \|\rho_k\|_{L^1(\Omega)} \le c t_k^\gamma
\]
for all $k$ such that $t_k<1$.
With Lemma \ref{lem_conv_aux1}, we get
\[%\begin{aligned}
  +\infty > \sum_{k=0}^\infty t_k \|\rho_k\|_{L^1(\Omega)}
  = \left( \sum_{k:\, t_k=1} \|\rho_k\|_{L^1(\Omega)} \right) + \left( \sum_{k:\, t_k<1} t_k\|\rho_k\|_{L^1(\Omega)} \right)
   \ge \left( \sum_{k:\, t_k=1} \|\rho_k\|_{L^1(\Omega)} \right) + c \left( \sum_{k:\, t_k<1} \|\rho_k\|_{L^1(\Omega)}^{1+\frac1\gamma} \right),
%\end{aligned}
\]
which results in $\lim_{k\to\infty} \|\rho_k\|_{L^1(\Omega)} = 0$.
Hence, the algorithm stops after finitely many iterations if $\dtol>0$.
\end{proof}

\section{Numerical results}
Let us now report about numerical experiments with Algorithm \ref{alg_top}.
Here, we consider the optimal control problem
\[
 J(y,u) := \frac12\|y-y_d\|_{L^2(\Omega)}^2 + \frac \alpha 2 \|u\|_{L^2(\Omega)}^2 + I_{\mathbb Z\cap [-b,b]}(u)
\]
over all $u\in L^2(\Omega)$ and $y\in H^1_0(\Omega)$ satisfying
\[\begin{aligned}
 -\Delta y &= u &&\text{ in } \Omega,\\
 y &= 0 &&\text{ on } \partial\Omega.
\end{aligned}\]
This fits into the setting of the paper with the choice
\[
 g(v) := \frac\alpha2 v^2 + I_{ \Z \cap [-b,b] }
\]
Here, we chose $\Omega=(0,1)^2$,
\[
 y_d(x_1,x_2) = 10 x_1 \sin(5x_1) \cos(7x_2), \quad \alpha = 0.01,\quad  \beta=0.01,\quad b = 10.
\]
We discretized the problem with piecewise linear finite elements on a regular mesh for state and adjoint variables, while
the control was discretized with piecewise constant finite elements.
We report the results for a sequence of different meshes, where the finest mesh has mesh-size $h=1.41 \cdot 10^{-3}$ resulting in
$\approx2 \cdot 10^6$ degrees of freedom for the control variables, which results in a mixed-integer optimization problem with $\approx2 \cdot 10^6$ integer variables.
In the implementation of Algorithm \ref{alg_top} a greedy strategy was used to determine $B_t$.
The loop in Algorithm \ref{alg_top} was terminated if in the inner loop $t|\Omega|$ was smaller than any of the elements in the grid.

Now let us report about some of the results.
The optimal control can be seen in the left plot of Figure \ref{fig1}.
In the right plot, we report about the iteration history of the residual $\|\rho_k\|_{L^1(\Omega)}$.
Surprisingly, the iterations seem to be mesh independent.
In addition, for this particular problem a very small number of iterations was needed to optimize over $2 \cdot 10^6$ discrete control variables.

\begin{figure}[h!]
 \includegraphics[width=0.48\textwidth]{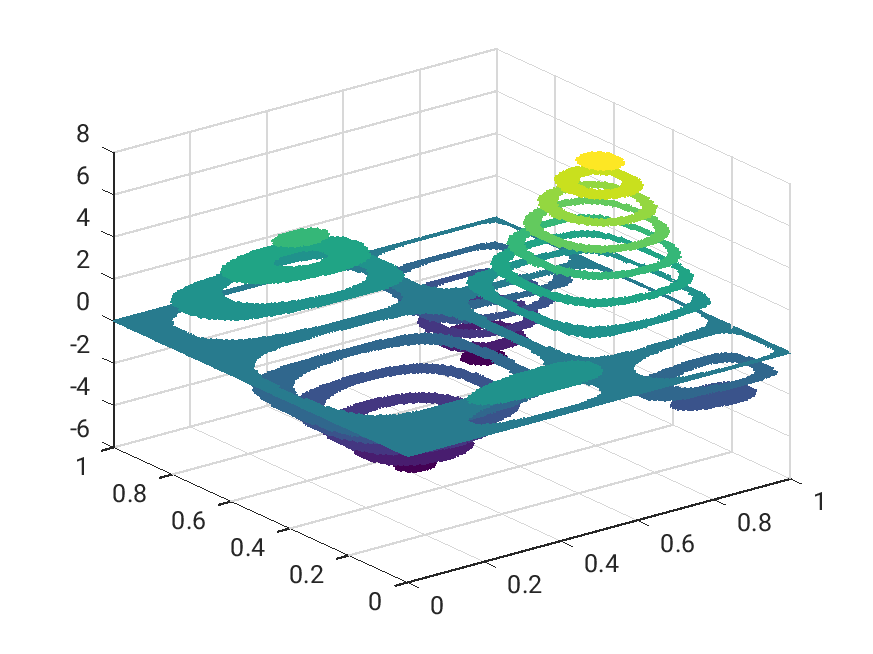}
 \includegraphics[width=0.48\textwidth]{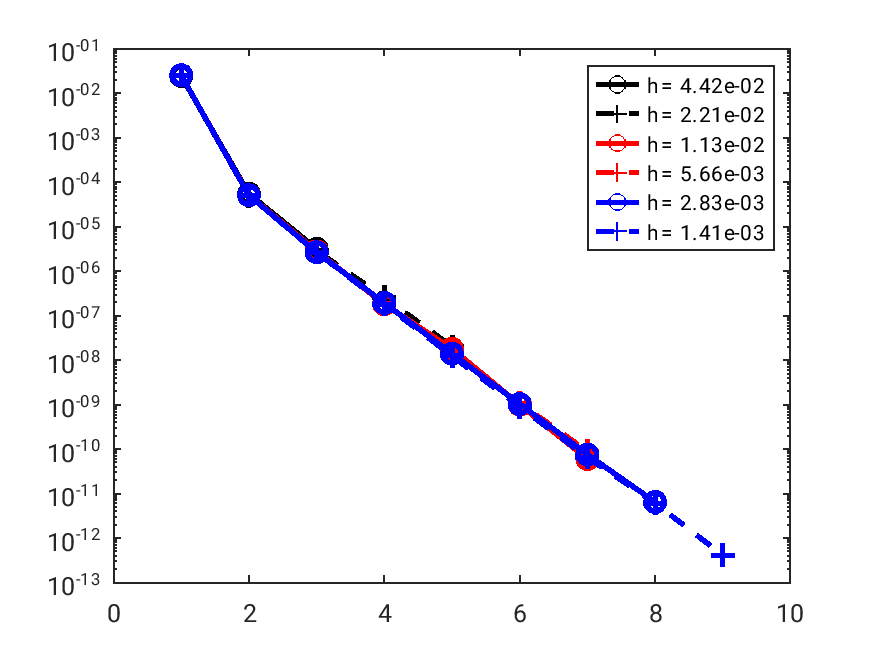}

  \caption{Optimal control (left), iteration history (right)}
  \label{fig1}
\end{figure}

This is underlined by the results in Table \ref{table1}. It shows for different discretizations the final value of the objective $J$ and the final value
of the residual $\|\rho\|_{L^1(\Omega)}$. As can be seen from the last column of this table, very few outer iterations are needed.
In conclusion, this new algorithm seems to be capable of solving quite challenging mixed-inter programs.

\begin{table}[h!]
\sisetup{table-alignment-mode = format, table-format = 2.2e2, table-number-alignment = left}
\begin{center}
\begin{tabular}{S[table-format = 1.2e2]S[table-format = 1.3]S[table-format=1.2e2]S[table-format=2] }
\toprule
$h$ & $J$ & $\|\rho\|_{L^1(\Omega)}$ & It\\
\midrule
4.42e-02 & 4.706 & 3.20e-06 & 4 \\
 2.21e-02 & 5.048 & 2.02e-08 & 6 \\
 1.13e-02 & 5.210 & 6.00e-11 & 8 \\
 5.66e-03 & 5.293 & 8.91e-11 & 8 \\
 2.83e-03 & 5.334 & 6.46e-12 & 9 \\
 1.41e-03 & 5.354 & 4.11e-13 & 10 \\
\bottomrule
\end{tabular}
\end{center}
\caption{Iteration history}
  \label{table1}
\end{table}

\section*{Acknowledgements}

This research was supported by the German Research Foundation (DFG) under grant number WA 3626/3-2
within the priority program ``Non-smooth and Complementarity-based Distributed Parameter
Systems: Simulation and Hierarchical Optimization'' (SPP 1962).
The author thanks Anna Lentz for comments on an earlier version of this manuscript.

% \bibliography{topological}
% \bibliographystyle{plainurl-abbrv}

\end{document}